\documentclass[11pt]{amsart}
\usepackage{amssymb, amsmath, latexsym,   bm, mathrsfs}

\renewcommand{\baselinestretch}{\baselinestretch}
\renewcommand{\baselinestretch}{1.1}
\numberwithin{equation}{section}

\newtheorem{thm}{Theorem}[section]
\newtheorem{lem}[thm]{Lemma}
\newtheorem{cor}[thm]{Corollary}
\newtheorem{prop}[thm]{Proposition}
\newtheorem{defn}[thm]{Definition}

\newcommand{\gen}{\text{gen}}

\newcommand{\ord}{\textnormal{ord}}

\newcommand{\z}{{\mathbb Z}}

\newcommand{\rank}{\textnormal{rank}}

\newcommand{\bx}{\bm x}
\newcommand{\by}{\bm y}
\newcommand{\bz}{\bm z}
\newcommand{\be}{\bm e}
\newcommand{\bff}{\bm f}
\newcommand{\bv}{\bm v}
\newcommand{\bu}{\bm u}
\newcommand{\bw}{\bm w}
\newcommand{\0}{\bm 0}

\newcommand{\fs}{\mathfrak s}
\newcommand{\ring}{\mathfrak o}
\newcommand{\ideal}{\mathfrak a}
\newcommand{\vol}{\mathfrak v}
\newcommand{\norm}{\textnormal{Nr}}
\newcommand{\trace}{\textnormal{Tr}}
\newcommand{\p}{\mathfrak p}
\newcommand{\qq}{\mathfrak q}
\newcommand{\pr}{\textnormal{pr}}

\begin{document}

\author{Wai Kiu Chan}
\address{Department of Mathematics and Computer Science, Wesleyan University, Middletown CT, 06459, USA}
\email{wkchan@wesleyan.edu}

\author{Byeong-Kweon Oh}
\address{Department of Mathematical Sciences and Research Institute of Mathematics,
Seoul National University,
 Seoul 151-747, Korea}
\email{bkoh@snu.ac.kr}
\thanks{This work of the second author was supported by the National Research Foundation of Korea (NRF-2019R1A2C1086347) and (NRF-2020R1A5A1016126).}



\title[Recover a quadratic form by representations]{Can we recover an integral quadratic form by representing  all its subforms?}

\begin{abstract}
Let $\ring$ be the ring of integers of a number field.   If $f$ is a quadratic form over $\ring$ and $g$ is another quadratic form over $\ring$ which represents all proper subforms of $f$, does $g$ represent $f$?  We show that if $g$ is indefinite, then $g$ indeed represents $f$.  However, when $f$ is positive definite and indecomposable, then there exists a $g$ which represents all proper subforms of $f$ but not $f$ itself.    Along the way we give a new characterization of positive definite decomposable quadratic forms over $\ring$ and a number-field generalization of the finiteness theorem of representations of quadratic forms by quadratic forms over $\z$ \cite[Theorem 3.3]{kko} which asserts that given any infinite set $\mathscr S$ of classes of positive definite integral quadratic forms over $\ring$ of a fixed rank, there exists a finite subset $\mathscr S_0$ of $\mathscr S$ with the property that a positive definite quadratic form over $\ring$ represents all classes in $\mathscr S$ if and only if it represents all classes in $\mathscr S_0$.
\end{abstract}

\maketitle

\section{Introduction}

Let $f$ be an integral quadratic form.  A quadratic form $f'$ is said to be represented by $f$, or $f'$ is a subform of $f$, if it is obtained from $f$ by a linear change of variables over the integers.  If the change of variables is invertible over the integers, then $f$ and $f'$ are said to be equivalent.  A proper subform of $f$ is a subform of $f$ which is not equivalent to $f$.  Along with their investigation of the minimal universality criterion sets for positive definite integral quadratic forms, the authors of \cite{klo} consider the following natural question:
\begin{quote}
$(*)$ \quad Let $f$ be a positive definite integral quadratic form.  If a positive definite quadratic form $g$ represents all the proper subforms of $f$, does $g$ represent $f$?
\end{quote}

At first sight this question should have an affirmative answer since a lot of the proper subforms of $f$ are almost equal to $f$.  However, in \cite[Section 4]{klo} it was shown that if $f$ is a positive definite indecomposable quadratic form in $\leq 4$ variables, then there is always a positive definite quadratic form $g$ which represents all proper subforms of $f$ but not $f$ itself.  This suggests that indecomposability plays a key role to the answer of the question $(*)$.

In this paper, we consider $(*)$ from a broader perspective by allowing the quadratic forms to be indefinite.  Our results show a great contrast between the definite and the indefinite cases.  In the indefinite case, we show that if $g$ is an indefinite quadratic form which represents all the proper subforms of $f$, then $g$ must represent $f$ (Theorem \ref{indefinitethm}).  However, in the definite case, we prove that for any positive definite indecomposable quadratic form $f$, regardless of the number of variables, there must be a positive definite quadratic which represents all the proper subforms of $f$ but not $f$ itself (Theorem \ref{main}).   In the terminology of \cite[Definition 3.1]{klo}, this means that a positive definite indecomposable quadratic form is not {\em recoverable}.

Our algebraic and arithmetic approach is flexible enough to let us expand the scope of the discussion to include quadratic forms over the ring of integers of a number field.  Furthermore, it will be more convenient to adopt the geometric language of quadratic spaces and lattices.  We will recall in the next section some of the notations and terminologies necessary in our discussion.  The readers can find a thorough treatment of quadratic spaces and lattices and any unexplained terminologies in \cite{om}.

For the indefinite case, the main tool is the powerful theory of spinor exceptions.  The interested readers are referred to \cite{hsx} and \cite{x} for a more recent account of this important theory concerning representations of indefinite lattices.  As for the definite case, there are two main ingredients both of which could be of independent interest to the readers.  The first one is the idea of a {\em nonsplit volume-minimizing flag} in a positive definite indecomposable lattice (Definition \ref{flag}).  This leads to a new characterization of positive definite decomposable $\ring$-lattices (Theorem \ref{decomposable}).  The second one is the number-field generalization of the finiteness theorem of representation of quadratic forms by quadratic forms in \cite[Theorem 3.3]{kko} (Theorem \ref{finite}).  An interesting consequence is the existence of a finite universality  criterion set for positive definite quadratic forms over a totally real number field.

\section{Preliminaries}

Let $\ring$ be a Dedekind domain of characteristic 0, and $F$ be its field of fractions.  A quadratic space over $F$ is a finite dimensional vector space over $F$ which is equipped with a quadratic map.  For the sake of convenience, the quadratic map on any quadratic space will be denoted by $Q$.  Associated to $Q$ is the bilinear map $B$  which satisfies the identity $2B(\bx, \by) = Q(\bx + \by)  - Q(\bx) - Q(\by)$ for all vectors $\bx$ and $\by$ in the underlying quadratic space.  If $\bx_1, \ldots, \bx_\ell$ are vectors in a quadratic space, the Gram matrix with respect to the ordered set $\{\bx_1, \ldots, \bx_\ell\}$ is the $\ell\times \ell$ matrix $(B(\bx_i, \bx_j))$.

An $\ring$-lattice is a finitely generated $\ring$-module $L$ in a quadratic space $V$ over $F$.  If $FL = V$, we say that $L$ is an $\ring$-lattice {\em on} $V$.  The {\em dual} of $L$, denoted $L^\#$, is the set $\{\bx \in FL : B(\bx, \bv) \in \ring \mbox{ for all } \bv \in L\}$ which is also an $\ring$-lattice on $FL$.  By the {\em scale} $\mathfrak s(L)$ of $L$ we mean the $\ring$-module generated by the subset $B(L, L)$ of $F$, and the {\em norm} $\mathfrak n(L)$ of $L$ is the $\ring$-module generated by the subset $Q(L)$ of $F$.  By their definitions, we see that $\mathfrak s(L) \supseteq \mathfrak n(L) \supseteq 2\mathfrak s(L)$.

We say that $L$ is an integral $\ring$-lattice if $\mathfrak s(L) \subseteq \ring$.  Let $\ideal$ be a fractional ideal of $\ring$.  Then, $L$ is called an $\ideal$-maximal $\ring$-lattice on $V$ if for every $\ring$-lattice $K$ on $V$ such that $L \subseteq K$ and $\mathfrak n(K) \subseteq \ideal$, then $L = K$.

When $L$ and $K$ are $\ring$-lattices in a quadratic space, $L + K$ is the their sum as $\ring$-modules which is an $\ring$-lattice.  If $L \cap K = \0$, then their sum becomes a direct sum which is denoted by $L \oplus K$.   Further, if $L \cap K = \0$ and $B(L, K) = \{0\}$, then their sum becomes an orthogonal sum and will be denoted by $L \perp K$.

As an module over $\ring$, $L$ needs not be free.  But there are linearly independent vectors $\bx_1, \ldots, \bx_\ell$ in $L$ and fractional ideals $\ideal_1, \ldots, \ideal_\ell$ such that $L = \ideal_1\bx_1 \oplus \cdots \oplus \ideal_\ell\bx_\ell$.  The volume of $L$, denoted $\vol(L)$, is defined as
$$\vol(L) = (\ideal_1\cdots \ideal_\ell)^2 \det(B(\bx_i, \bx_j)).$$
If $L$ is integral, then it is nondegenerate if $\vol(L) \neq 0$, and is unimodular if $\vol(L) = \ring$.  More generally, $L$ is $\ideal$-modular if $\fs(L) = \ideal$ and $\vol(L) = \ideal^\ell$, which is equivalent to the equality $\ideal L^\# = L$.  Since $(L^\#)^\# = L$, we see that $L$ is $\ideal$-modular if and only if $L^\#$ is $\ideal^{-1}$-modular.

Let $N$ be another $\ring$-lattice.  We say that $N$ is represented by $L$, denoted $N \longrightarrow L$, if there exists an $\ring$-linear map $\varphi: N \longrightarrow L$ such that $Q(\varphi(\bx)) = Q(\bx)$ for all $\bx \in N$.   Such a map $\varphi$ is called a representation from $N$ to $L$.   It is called primitive if $\varphi(N)$ is a  primitive sublattice (or a direct summand) of $L$, which is the same as $\varphi(FN) \cap L = \varphi(N)$.  An injective representation from $N$ to $L$ is called an isometry from $N$ into $L$, and $N$ and $L$ are said to be isometric if there is an isometry from $N$ onto $L$.  It is known that if $N$ is nondegenerate, then a representation $\varphi: N \longrightarrow L$ must be an isometry.

The isometry class of $L$, denoted by $[L]$, is the set of $\ring$-lattices that are isometric to $L$.  If $N$ is represented by $L$, then any lattices in $[N]$ will also be represented by any lattices in $[L]$.  Therefore, we may speak of an isometry class of lattices being represented by another isometry class of lattices.

We point out in here that when $\ring$ is the ring of integers in a number field, the set of isometry class of $\ring$-lattices is countable.  Indeed, every $\ring$-lattice of rank $n$ must be isometric to an $\ring$-lattice on $F^n$ equipped with a quadratic map which can be identified with an $n\times n$ symmetric matrix $S$ with entries in $F$.  As is indicated earlier, to specify such an $\ring$-lattice on $F^n$ all we need is an $n\times n$ matrix $A$ over $F$ whose columns are linearly independent and an ordered $n$-tuple $I$ of fractional ideals of $F$.  Since $F$ is a number field, there are countably many these $S$, $A$, and $I$.  Therefore, all together there are only countably many tuples $(S, A, I)$, and hence there are only countably many isometry classes of $\ring$-lattices of rank $n$.

Suppose that $N$ is a sublattice of $L$.  We say that $N$ splits $L$ if $N$ is an orthogonal summand of $L$.   The lattice $L$ is called decomposable if $L = N \perp M$ for some nonzero sublattices $N$ and $M$; otherwise it is indecomposable.  It is clear that every $\ring$-lattice $L$ is the orthogonal sum $L = L_1 \perp \cdots \perp L_t$ of indecomposable sublattices $L_1, \ldots, L_t$.  When $\ring$ is the ring of integers in a totally real number field and $L$ is positive definite, Eichler's unique decomposition theorem \cite{ei,kn} says that the indecomposable sublattices $L_1, \ldots, L_t$  are uniquely determined by $L$ (aside from their order).  O'Meara \cite{omd} showed that when $F$ is totally real, positive definite indecomposable $\ring$-lattices exist in all ranks, and if the class number of $F$ is 1 then there are positive definite indecomposable $\ring$-lattices of any given rank and any given discriminant that is not square-free with exactly three exceptions, all of them over $\z$.

Suppose that $\ring$ is the ring of integers in a number field $F$.  For every finite prime $\p$ of $F$, $F_\p$ and $\ring_\p$ denote the completion of $F$ and $\ring$ at $\p$, respectively, and $L_\p$ denotes the completion of $L$ at $\p$ which is an $\ring_\p$-lattice.   The genus of an $\ring$-lattice $L$, denoted $\gen(L)$, is the set of $\ring$-lattices $K$ on $FL$ such that $K_\p$ is isometric to $L_\p$ for all finite primes $\p$.  An $\ring$-lattice $N$ is said to be represented by $\gen(L)$ if $N$ is represented by some $\ring$-lattice in $\gen(L)$.  This is the same as saying that $N_\p$ is represented by $L_\p$ for all primes $\p$.

\section{The Indefinite Case}

In this section, $F$ is a number field and $\ring$ is the ring of integers in $F$.  All $\ring$-lattices are assumed to be nondegenerate and integral.

\begin{prop} \label{indefinite1}
Let $L$ be an $\ring$-lattice.  Suppose that $M$ is an $\ring$-lattice which represents all proper sublattices of $L$.  Then $L$ is represented by the genus of $M$.  If, in addition, $L$ and $M$ have the same rank, then $L$ is represented by $M$.
\end{prop}
\begin{proof}
Since $M$ represents every proper sublattice of $L$, $F_\p M$ represents $F_\p L$ for all infinite primes $\p$ of $F$.  Now, suppose that $\p$ is a finite prime of $F$.   Let $L'$ be a proper sublattice of $L$ such that $L'_\qq = L_\qq$ for all finite primes $\qq \neq \p$. Then, $L'$ is represented by $M$, and hence $L_\qq = L'_\qq$ is represented by $M_\qq$ for all $\qq \neq \p$; the existence of such $L'$ is an immediate consequence of \cite[81:14]{om}.   Applying the same argument with a finite prime different from $\p$, we may conclude that $L_\p$ is represented by $M_\p$ as well.

Suppose in addition that $L$ and $M$ have the same rank.  Since $F_\p L$ is isometric to $F_\p M$ for all primes $\p$, $FL$ is isometric to $FM$ by Hasse-Minkowski Principle.  Hence, we may assume that $L$ and $M$ are $\ring$-lattices on the same space.   Let $V$ be the space underlying both $L$ and $M$.   Choose a finite nondyadic prime $\p$ such that both $L_\p$ and $M_\p$ are unimodular.  Let $\{\bx_1, \ldots, \bx_\ell\}$ be an orthogonal basis of $L_\p$, and by \cite[81:14]{om} let $N$ be the proper sublattice of $L$ such that $N_\qq = L_\qq$ for all $\qq \neq \p$ and $N_\p = \ring_\p \bx_1 + \cdots + \ring_\p \bx_{\ell-1} + \ring_\p \pi \bx_\ell$, where $\pi$ is a uniformizer at $\p$.   By the hypothesis, $N$ is represented by $M$.  We may assume that $N \subseteq M$.  Then, both $L_\p$ and $M_\p$ are unimodular $\ring_\p$-lattices containing the unimodular sublattice $K_{(\p)}: = \sum_{i=1}^{\ell - 1} \ring_\p \bx_i$ of co-rank 1 as an orthogonal summand.   The orthogonal complement of $F_\p K_{(\p)}$ in $V_\p$ is the 1-dimensional quadratic space $F_\p \bx_\ell$ which contains only one unimodular $\ring_\p$-lattice, namely $\ring_\p\bx_\ell$.  This implies that  $L_\p = M_\p$.  For any finite prime $\qq \neq \p$, $L_\qq = N_\qq \subseteq M_\qq$.  Therefore, $L \subseteq M$.
\end{proof}

In what follows, $M$ is an indefinite $\ring$-lattice of rank $\geq 3$, and $L$ is another $\ring$-lattice in $FM$ which is represented by the genus of $M$.   The theory of spinor exceptions will play an important role in the following discussion.  It is known that $L$ may not be represented by every class in $\gen(M)$; in that case and when $\rank(M) \geq 3$,  $L$ is called a spinor exception of $\gen(M)$.   Following the notations used in \cite{hsx}, associated to $M$ and $L$ is a subgroup $F^\times \theta(G_{M/L})$ (or $F^\times \theta_A(\gen(M) : L)$ used in \cite{x}) of the ideles of $F$.  Via class field theory, $F^\times \theta(G_{M/L})$ corresponds to an abelian extension $\Sigma_{M/L}$ of $F$ which is unramified outside of the set of prime divisors of $2\vol(M)$.  When $\rank(M) \geq 3$, whether or not $M$ represents $L$ largely depends on the module index ideal $[L : L\cap M]$ and its image in $\text{Gal}(\Sigma_{M/L}/F)$ under the Artin map.    The readers are encouraged to consult \cite{cx, hsx, x} for more details.  We will describe briefly the results used in later discussion.

For any $\ring$-lattices $L_1$ and $L_2$  of the same rank such that $L_1 \subseteq L_2$, let $[L_2 : L_1]$ be the module index ideal; see \cite[page 10]{cf} or \cite[page 94]{ft}.  It can be shown--for example, using \cite[Theorem 14(b)]{ft}--that $[L_2 : L_1]$ is the product of the annihilators of the simple factors of a composition series of $L_2/L_1$.  In particular, if $L_2/L_1 \cong \ring/\p$ where $\p$ is a finite prime, then $[L_2 : L_1] = \p$.

Suppose that $[L : L\cap M]$ is relatively prime to $2\vol(M)$.  This is not a serious restriction; as a matter of fact, by the weak approximation of the special orthogonal group there is always an $\ring$-lattice $M'$ in the isometry class of $M$ such that $L_\p \subseteq M'_\p$ for all prime divisors $\p$ of $2\vol(M)$.  In any case, it follows from \cite[Theorem 5.3]{x} that $L$ is represented by $M$ exactly when the image of $[L : L\cap M]$ under the Artin map is trivial in $\text{Gal}(\Sigma_{M/L}/F)$.

\begin{thm} \label{indefinitethm}
Let $L$ be an $\ring$-lattice.  If $M$ is an indefinite $\ring$-lattice which represents all proper sublattices of $L$, then $M$ represents $L$.
\end{thm}
\begin{proof}
For the sake of convenience, let $\ell$ be the rank of $L$ and $m$ be the rank of $M$.

We first assume that $m \geq 3$.  If $m - \ell \geq 3$, then $M$ represents $L$ by \cite[page 135]{hsx}.  Suppose that $m - \ell \leq 2$.  We may assume that $FL$ is a subspace of $FM$.   Let $\p$ be a finite prime which does not divide $2\vol(M)$ and splits completely in the extension $\Sigma_{M/L}$.  Let $K$ be a proper sublattice of $L$ such that $L/K \cong \ring/\p$ as an $\ring$-module.  Then, $K$ is represented by $M$ by the hypothesis.  We may then assume that $K \subseteq M$.  If $L \subseteq M$, then we are done.  Otherwise, $L\cap M = K$ and $L/L\cap M \cong \ring/\p$, i.e. $[L : L\cap M] = \p$.  Since $\p$ splits completely in $\Sigma_{M/L}$, its image under the Artin map is trivial in $\text{Gal}(\Sigma_{M/L}/F)$.    By \cite[Theorem 5.3]{x}, $L$ is represented by $M$.

In view of Proposition \ref{indefinite1}, we are left with the case when $m = 2$ and $\ell = 1$.  So,  suppose that $M$ is an indefinite binary $\ring$-lattice and $L$ is a unary $\ring$-lattice.  We first assume that $M$ is anisotropic.  In this case, let $\p$ be a finite prime such that $\p \nmid 2\vol(L)\vol(M)$ and $M_\p$ is anisotropic.  Let $N$ be a proper sublattice of $L$ such that $N_\qq = L_\qq$ for all $\qq \neq \p$.   Then, $N$ is represented by $M$ and we may assume that $N \subseteq M$.  At $\p$, $L_\p \subseteq M_\p$ because $Q(L_\p) \subseteq \ring_\p$ and $M_\p$ is the unique $\ring_\p$-maximal lattice on $F_\p M$.  But for all $\qq \neq \p$, $L_\qq = N_\qq \subseteq M_\qq$; whence $L \subseteq M$.

Now suppose that $M$ is a binary isotropic $\ring$-lattice.  The underlying quadratic space $FM$ is a hyperbolic plane.  Let $\be$ and $\bm f$ be two linearly independent isotropic vectors in $FM$.  For almost all finite primes $\p$,  $M_\p = \ring_\p\be + \ring_\p\bff$ and $L_\p = \ring_\p \bv$ for some $\bv \in L$.  We choose one such $\p$ that is nondyadic and principal, and let $\pi$ be a generator of $\p$.  Let $L'$ be the sublattice of $L$ such that $L'_\qq = L_\qq$ for all $\qq \neq \p$ and $L'_\p = \ring_\p\pi\bv$.  It is represented by $M$, and once again we may assume that it is a sublattice of $M$.  If $\pi\bv$ is not a primitive vector in $M_\p$, then $\bv \in M_\p$ and $L \subseteq M$.  So, let us assume that $\pi\bv = \alpha\be + \beta\bff$, where $\alpha$ is a unit of $\ring_\p$.  Then, $\pi^2 \mid \beta$ because $2\alpha\beta = Q(\pi\bv)$ and $\pi^2 \mid Q(\pi\bv)$.  Let $\varphi$ be the isometry on $FM$ such that $\varphi(\be) =\pi\be$ and $\varphi(\bff) = \pi^{-1}\bff$.  Then $\varphi(L') \subseteq M$ and $\varphi(\pi\bv) = \pi\alpha \be + \pi^{-1}\beta \bff$.  So, $\varphi(L) \subseteq M$ and we are done.
\end{proof}

\section{The Definite Case}

In this section, $F$ is assumed to be a totally real number field and the ring $\ring$ is always the ring of integers in $F$.  We denote the norm and trace from $F$ to $\mathbb Q$ by $\norm$ and $\trace$, respectively.  We continue to assume that all $\ring$-lattices are nondegenerate and integral.

\subsection{Nonsplit volume-minimizing flags}

Let $L$ be a positive definite $\ring$-lattice of rank $\ell$.    For any positive integer $i < \ell$, let $\mathscr P^i(L)$ be the set of primitive sublattices of $L$ of rank $i$.  Let $N$ be another lattice of rank $n < \ell$.  The set of primitive representations from $N$ to $L$ is denoted by $R^*(N, L)$.  Define
$$\mathscr D(N) := \{M\in \mathscr P^{n+1}(L) : \exists\, \varphi \in R^*(N,L),  \varphi(N) \subseteq M \mbox{ does not split $M$} \}.$$
This set could be empty, even when $N$ is a primitive sublattice of $L$.  For example, if $N$ is a unimodular sublattice of $L$, then $\varphi(N)$ must split $L$ for every $\varphi \in R^*(N,L)$  and $\mathscr D(N)$ is empty.  On the other hand, if $L$ is indecomposable, then $\mathscr D(N)$ is never empty for any primitive proper sublattice $N$ of $L$.  This can be seen as follows.  There is a nonzero sublattice $K$ of $L$ such that $L = N \oplus K$.  This sublattice $K$ cannot be orthogonal to $N$; otherwise $L$ is decomposable which is not the case.   Therefore, there must be a vector $\bx \in K$ and a fractional ideal $\ideal$ such that $B(\bx, N) \neq \{0\}$ and $N \oplus \ideal \bx$ is a primitive sublattice of $L$ of rank $n + 1$.  This means that $N\oplus \ideal \bx$ is in $\mathscr D(N)$.

Let
$$\mathscr M(N) = \left\{ M \in \mathscr D(N) : \norm(\vol(M)) \leq \norm(\vol(M')) \mbox{ for all } M' \in \mathscr D(N) \right\}.$$
We extend this definition to include the case $N = \0$  by setting $\mathscr M(\0)$ to be the set of rank-1 sublattices $M$ of $L$ such that $\norm(\vol(M))$ is the smallest.  It is easy to see that all sublattices in $\mathscr M(\0)$ are in $\mathscr P^1(L)$.  Since there are only finitely many sublattices of $L$ of a fixed rank and a fixed volume, each $\mathscr M(N)$ is a finite set.  Moreover, $\mathscr M(N) = \{L\}$ for any sublattice $N$ of rank $\ell -1$ which does not split $L$.  Note that when $L$ is indecomposable, $\mathscr M(N)$ is not empty because $\mathscr D(N)$ is not empty.

\begin{defn} \label{flag}
Let $L$ be a positive definite $\ring$-lattice of rank $\ell$, and $k \leq \ell$ be a positive integer.  A nonsplit  volume-minimizing flag of $L$ of length $k$ is a chain of sublattices of $L$
$$\0 = N_0 \subsetneq  N_1 \subsetneq N_2 \subsetneq \cdots \subsetneq N_{k}$$
such that $N_i \in \mathscr M(N_{i-1})$ for $1 \leq i \leq k$.  Such a flag is called maximal if $k = \ell$.
\end{defn}

\begin{lem} \label{primrep}
Suppose that $\0 = N_0 \subsetneq N_1 \subsetneq \cdots \subsetneq N_k$ is a nonsplit  volume-minimizing flag of a positive definite $\ring$-lattice $L$. Then, for $1 \leq j \leq k$, every representation of $N_j$ by $L$ must be primitive.
\end{lem}
\begin{proof}
It is clear that the assertion holds for $N_1$.   Suppose that $k > 1$ and  there is a positive integer $j \leq k-1$ such that the assertion holds for $N_1, \ldots, N_j$ but not for $N_{j+1}$.  Then, there is a representation $\psi$ of $N_{j+1}$ by $L$ such that $\psi(N_{j+1})$ is not a primitive sublattice of $L$.

Let $\tilde{N}_{j+1}$ be the sublattice $\psi(FN_{j+1})\cap L$.  It is a primitive sublattice of $L$ containing $\psi(N_{j+1})$ as a proper sublattice and $\norm(\vol(\tilde{N}_{j+1})) < \norm(\vol(N_{j+1}))$ as a result.  Since the restriction of $\psi$ on $N_j$ must be a primitive representation from $N_j$ to $L$  and $\psi(N_j) \subset \tilde{N}_{j+1}$, it follows that $\psi(N_j)$ must split $\tilde{N}_{j+1}$.  So,
\begin{equation} \label{1}
\tilde{N}_{j+1} = \psi(N_j) \perp \mathfrak b \bz
\end{equation}
for some $\bz \in L$ and  fractional ideal $\mathfrak b$ containing $\ring$.  Since $N_j$ is primitive in $N_{j+1}$, there exist $\bv \in L$ and a fractional ideal $\mathfrak a$ containing  $\ring$ such that
$$N_{j+1} = N_j \oplus \mathfrak a \bv.$$
Then, $\bv \in N_{j+1}$, and hence there exist $\beta \in \mathfrak b$ and $\by \in N_j$ such that $\psi(\bv) = \psi(\by) + \beta \bz$.  For any $\alpha \in \mathfrak a$, $$\psi(\alpha \bv) = \alpha \psi(\by) + \alpha\beta\bz = \psi(\alpha\by) + \alpha\beta\bz.$$
Since $\alpha \bv$ is in $N_{j+1}$ and $\psi(N_{j+1}) \subseteq \tilde{N}_{j+1}$, by \eqref{1} we can write $\psi(\alpha\bv) = \be + \bff$ with $\be \in \psi(N_j)$ and $\bff \in \mathfrak b \bz$.  Since $F\psi(N_j) \cap F\bz = \0$, we must have $\psi(\alpha \by) = \be \in \psi(N_j)$ and hence $\alpha\beta \bz \in \psi(N_{j+1})$.

Thus, on one hand, we have
\begin{eqnarray*}
\psi(N_{j+1}) = \psi(N_j) \oplus \mathfrak a \psi(\bv) & = & \psi(N_j) \oplus \mathfrak a(\psi(\by) + \beta \bz)\\
    &  \subseteq & \psi(N_j) + \mathfrak a\psi(\by) + \mathfrak a\beta\bz \\
    & \subseteq & \psi(N_j) \perp \mathfrak a \beta\bz.
\end{eqnarray*}
On the other hand, $\psi(N_j) \subseteq \psi(N_{j+1})$ and $\alpha\beta \bz \in \psi(N_{j+1})$ for all $\alpha \in \ideal$.  Hence
$$\psi(N_{j+1}) = \psi(N_j) \perp \mathfrak a\beta\bz.$$
This means that $N_j$ splits $N_{j+1}$ which is impossible.  This contradiction proves the lemma.
\end{proof}

\subsection{Indecomposable lattices}

Recall that an $\ring$-lattice is indecomposable if it is not the orthogonal sum of two nonzero proper sublattices.

\begin{prop}
Every positive definite indecomposable $\ring$-lattice has a maximal nonsplit volume-minimizing flag.
\end{prop}
\begin{proof}
Let $L$ be a positive definite indecomposable $\ring$-lattice of rank $\ell$.  We start with building the flag by fixing a choice of $N_1$ from $\mathscr M(\0)$.  Suppose that for some positive integer $k \leq \ell - 1$, we have constructed a nonsplit volume-minimizing flag of length $k$:
$$\0 = N_0 \subsetneq N_1 \subsetneq \cdots \subsetneq N_k.$$
As is mentioned earlier, $\mathscr M(N_k)$ is nonempty because $L$ is indecomposable.  Let $M$ be a lattice in $\mathscr M(N_k)$, which must be primitive of rank $k + 1$, and $\varphi$ be a primitive representation of $N_k$ by $L$ such that $\varphi(N_k) \subseteq M$ and that $\varphi(N_k)$ does not split $M$.  It is clear that $\varphi(N_j) \in \mathscr M(\varphi(N_{j-1}))$ for $1 \leq j \leq k$ and $\mathscr M(N_k) = \mathscr M(\varphi(N_k))$.  We then have a nonsplit volume-minimizing flag of length $k +1$:
$$\0 = N_0 \subsetneq \varphi(N_1) \subsetneq \cdots \subsetneq \varphi(N_k) \subsetneq M.$$
Thus, we may construct inductively a maximal nonsplit volume-minimizing flag of $L$.
\end{proof}

\begin{thm} \label{decomposable}
A positive definite $\ring$-lattice $L$ is decomposable if and only if there are proper sublattices $L_1, \ldots, L_t$ of $L$ such that $L$ is represented by $L_1 \perp \cdots \perp L_t$.
\end{thm}
\begin{proof}
The ``only if" part of the statement is trivial.  So, we assume on the contrary that $L$ is indecomposable and  $\varphi: L \longrightarrow L_1 \perp \cdots \perp L_t$ is a representation with proper sublattices $L_1, \ldots, L_t$ of $L$.  Obviously, $t$ is at least 2.   Let
$$\0 = N_0 \subsetneq N_1 \subsetneq \cdots \subsetneq N_\ell = L$$
be a maximal nonsplit volume-minimizing flag of $L$.  We select vectors $\bv_1, \ldots, \bv_\ell$ and fractional ideals $\mathfrak a_1, \ldots, \mathfrak a_\ell$ containing $\ring$ such that
$$N_j = \mathfrak a_1\bv_1 \oplus \cdots \oplus \mathfrak a_j\bv_j$$
for $1 \leq j \leq \ell$.

Suppose that $\varphi(N_1)$ is not a sublattice of any of the $L_i$.  Then, $\varphi(\bv_1)$ is not in any of the $L_i$.   Without loss of generality, we may assume that $\varphi(\bv_1) = \be_1 + \bm f$, where $\0 \neq \be_1 \in L_1$ and $\0 \neq \bm f \in L_2\perp \cdots \perp L_t$.  Then, for any $\alpha_1 \in \mathfrak a_1$, $\alpha_1 \bv_1 \in N_1$ and $\varphi(\alpha_1\bv_1) = \alpha_1\be_1 + \alpha_1\bm f$ with $\alpha_1\be_1 \in FL_1$, hence $\alpha_1\be_1 \in L_1 \subseteq L$.  So, $\mathfrak a_1\be_1$ is a sublattice of $L$ and
\begin{eqnarray*}
\norm(\vol(N_1)) = \norm(\mathfrak a_1^2)\norm(Q(\bv_1)) & = & \norm(\mathfrak a_1^2)\norm(Q(\be_1) + Q(\bm f))\\
    & > & \norm(\mathfrak a_1^2) \norm(Q(\be_1)) \\
    & = & \norm(\vol(\mathfrak a_1\be_1)).
\end{eqnarray*}
This contradicts that $N_1$ is in $\mathscr M(\0)$.   So, we may assume that $\varphi(\bv_1) \in L_1$ and $\varphi(N_1)$ is a sublattice of $L_1$.

Clearly, $\varphi(N_\ell) = \varphi(L)$ is not a sublattice of $L_1$.  Thus,  there is a positive integer $k \leq \ell - 1$ such that $\varphi(N_j) \subseteq L_1$ for $1 \leq j \leq k$ but $\varphi(N_{k+1}) \nsubseteq L_1$.  Then, $\mathfrak a_{k+1}\varphi(\bv_{k+1}) \nsubseteq L_1$.  Suppose that $\varphi(\bv_{k+1}) \in L_1$.  For any $\alpha \in \ideal_{k+1}$, $\alpha \bv_{k+1} \in N_{k+1}$ and $\varphi(N_{k+1}) \subseteq L_1 \perp \cdots \perp L_t$,  $\alpha\varphi(\bv_{k+1}) = \varphi(\alpha\bv_{k+1})$  must be inside of $L_1$.  Therefore, $\ideal_{k+1}\varphi(\bv_{k+1}) \subseteq  L_1$ which is a contradiction.  This means that $\varphi(\bv_{k+1}) \not \in L_1$.

Thus,
$$\varphi(\bv_{k+1})  = \bx + \by, \quad \bx \in L_1,\quad \0 \neq \by \in L_2\perp \cdots \perp L_t.$$
For any $\alpha \in \ideal_{k+1}$, $\alpha \bv_{k+1} \in N_{k+1}$ and
$$\varphi(\alpha\bv_{k+1}) = \alpha \bx + \alpha \by,$$
with $\alpha \bx \in FL_1$ and $\alpha \by \in F(L_2 \perp \cdots \perp L_t)$.  Therefore, $\alpha \bx \in L_1$ and $\alpha \by \in L_2\perp \cdots \perp L_t$, whence  $\mathfrak a_{k+1}\bx \subseteq L_1$ and $\mathfrak a_{k+1}\by \subseteq L_2\perp \cdots \perp L_t$.  Moreover, $\bx \not \in F\varphi(N_k)$; otherwise $\mathfrak a_{k+1}\bx \subseteq F\varphi(N_k)\cap L_1 = \varphi(N_k)$ because $\varphi \in R^*(N_k, L)$ by Lemma \ref{primrep}, meaning that $\varphi(N_{k+1}) = \varphi(N_k) \perp \mathfrak a_{k+1}\by$ which is impossible.  Thus, $\varphi(N_k) \oplus \mathfrak a_{k+1}\bx$ is a sublattice of $L_1$ (hence of $L$ as well) of rank $k + 1$ containing $\varphi(N_k)$.  Let
$$\hat{N}_{k+1} = F(\varphi(N_k) \oplus \mathfrak a_{k+1}\bx) \cap L$$
which is a primitive sublattice of $L$.

Since $B(\varphi(\bv_i), \by) = 0$ for $1 \leq i \leq k$, the Gram matrix of the linearly independent set $\{\varphi(\bv_1), \ldots, \varphi(\bv_k), \varphi(\bv_{k+1})\}$ is
\begin{equation} \label{Gram}
\begin{bmatrix}
B(\varphi(\bv_1), \varphi(\bv_1)) & \cdots & \cdots & B(\varphi(\bv_1), \varphi(\bv_k)) & B(\varphi(\bv_1), \bx) \\
\vdots   &         &        & \vdots        & \vdots \\
\vdots  &           &          & \vdots     & \vdots\\
B(\varphi(\bv_k), \varphi(\bv_1)) & \cdots & \cdots & B(\varphi(\bv_k), \varphi(\bv_k)) & B(\varphi(\bv_k), \bx) \\
B(\bx, \varphi(\bv_1)) & \cdots & \cdots & B(\bx, \varphi(\bv_k)) & Q(\bx) + Q(\by)
\end{bmatrix}
\end{equation}

Let $\bm X$ and $\bm Y$ be the Gram matrices of the two sets of linearly independent vectors $\{\varphi(\bv_1), \ldots, \varphi(\bv_k), \bx\}$ and $\{\varphi(\bv_1), \ldots, \varphi(\bv_k), \by\}$, respectively.  Both $\bm X$ and $\bm Y$ are positive definite.  By expanding along the last row, we see that the determinant of the Gram matrix \eqref{Gram} is $\det(\bm X) + \det(\bm Y)$.
Thus,
\begin{eqnarray*}
\norm(\vol(N_{k+1})) & = & \norm(\vol(\varphi(N_{k+1})))\\
     & = & \norm(\mathfrak a_1^2 \cdots \mathfrak a_{k+1}^2)\cdot \norm(\det(\bm X) + \det(\bm Y))\\
    & > & \norm(\mathfrak a_1^2 \cdots \mathfrak a_{k+1}^2)\cdot \norm(\det(\bm X))\\
    & = & \norm(\vol(\varphi(N_k)\oplus \mathfrak a_{k+1}\bx))\\
    & \geq & \norm(\vol(\hat{N}_{k+1})).
\end{eqnarray*}
So, $\varphi(N_k)$ must split $\hat{N}_{k+1}$, that is, $\hat{N}_{k+1} = \varphi(N_k) \perp \mathfrak b \bz$ for some $\bz \in L$ and fractional ideal $\mathfrak b \supseteq \ring$.  

Since $\varphi(N_k) \oplus \mathfrak a_{k+1}\bx \subseteq \hat{N}_{k+1}$, we may write
$$\bx = \varphi(\bu) + \beta \bz, \quad \bu \in N_k, \quad \beta \in \mathfrak b.$$
For any $\alpha \in \ideal_{k+1}$, $\alpha \bx  = \alpha \varphi(\bu) + \alpha \beta \bz \in \hat{N}_{k+1}$, while $\alpha \varphi(\bu) \in F\varphi(N_k)$ and $\alpha\beta \bz \in F\bz$.  Therefore, $\alpha \varphi(\bu) \in \varphi(N_k)$ and $\alpha\beta \bz \in \mathfrak b\bz$ for all $\alpha \in \ideal_{k+1}$, whence $\mathfrak a_{k+1}\varphi(\bu) \subseteq \varphi(N_k)$ and $\mathfrak a_{k+1}\beta \bz \subseteq \mathfrak b\bz$.  Consequently,
\begin{eqnarray*}
\varphi(N_{k+1})  = \varphi(N_k) \oplus \mathfrak a_{k+1}\varphi(\bv_{k+1}) & = & \varphi(N_k) \oplus \mathfrak a_{k+1}(\bx + \by)\\
    & = & \varphi(N_k) \oplus \mathfrak a_{k+1}(\varphi(\bu) + \beta \bz + \by)\\
    & = & \varphi(N_k) \perp \mathfrak a_{k+1}(\beta \bz + \by).
\end{eqnarray*}
This is impossible since $N_k$ does not split $N_{k+1}$.  This completes the proof of the theorem.
\end{proof}


As a corollary to Theorem \ref{decomposable}, we have

\begin{thm} \label{main}
Let $L$ be a positive definite  indecomposable $\ring$-lattice.  Then, there is a positive definite  $\ring$-lattice which represents all proper sublattices of $L$ but not $L$ itself.
\end{thm}
\begin{proof}
By Theorem \ref{finite},  the number-field version of \cite[Theorem 3.3]{kko} which will be proved in the next section, there exist finitely many proper sublattices $L_1, \ldots, L_t$ of $L$ with the property that if a positive definite $\ring$-lattice represents every $L_i$, then it must represent all the proper sublattices of $L$.  As a result, $L_1\perp \cdots \perp L_t$ must represent all the proper sublattices of $L$.  However, by Theorem \ref{decomposable}, it does not represent $L$.
\end{proof}

\section{Finite universality  criterion sets}

We continue to assume that $F$ is a totally real number field.  Let $\mathscr S$ be an infinite set of classes of positive definite $\ring$-lattices of a fixed rank $n$.  A positive definite $\ring$-lattice is called $\mathscr S$-universal if it represents all classes in $\mathscr S$. The main result of this section is an extension of \cite[Theorem 3.3]{kko}: there exists a finite subset $\mathscr S_0$ of $\mathscr S$ such that if a positive definite $\ring$-lattice represents all the classes of $\ring$-lattices in $\mathscr S_0$, then it will represent all the classes of $\ring$-lattices in $\mathscr S$.  Following \cite{co} and \cite{ekk}, we call such a finite subset $\mathscr S_0$ an $\mathscr S$-universality  criterion set.  Our proof follows the same strategy employed in \cite{kko}.  However, extra care and additional effort are needed when we replace Minkowski reduction with Humbert reduction in the argument.  Humbert reduction does not yield a ``reduced basis" of an $\ring$-lattice even if the class number of $F$ is 1; it only gives a set of linearly independent vectors in the $\ring$-lattice which spans a sublattice whose index is bounded above by a constant depending only on $F$ and the rank of the $\ring$-lattice.  Lifting a representation of this sublattice to the $\ring$-lattice is what we need to overcome.

\subsection{Humbert Reduction}
The standard reference for Humbert reduction is Humbert's original paper \cite{Humbert}.   Let $N$ be a positive definite $\ring$-lattice of rank $n$.   In \cite{Humbert}, Humbert constructed linearly independent vectors $\bv_1, \ldots, \bv_n$ which play the roles of the vectors attaining the successive minima in Minkowski's reduction theory.  It may be worth mentioning that Humbert worked with matrices, i.e. free $\ring$-lattices.  In other words, he considered quadratic forms on $\ring^n$, and the vectors $\bv_1, \ldots, \bv_n$ his construction yields are vectors in $\ring^n$. In that setting,  he called the $n\times n$ matrix $[\bv_1\,\cdots\, \bv_n]$ a {\em matrixe auxiliaire}.  However, as is remarked in \cite[page 139]{hkk}, Humbert's method works equally well with non-free $\ring$-lattices.  As a matter of fact, there is always a free sublattice of $N$ of the same rank and whose index in $N$ is bounded above by a constant depending only on $F$.  We may then apply Humbert's argument to this free sublattice of $N$.

We collect below some useful properties of the set $\{\bv_1, \ldots, \bv_n\}$:
\begin{enumerate}
\item[(H1)]  By \cite[Th\'{e}or\`{e}me 1]{Humbert},  the index of the free sublattice $\sum \ring \bv_i$ in $N$ is bounded above by a constant depending only on $F$ and $n$.

\item[(H2)] By the construction given in \cite[\S 3]{Humbert}, $\trace(Q(\bv_1)) = \min\{\trace(Q(\bx)) : \0 \neq \bx \in N\}$, and for $i \geq 2$, we have $\trace(Q(\bv_i)) \leq \trace(Q(\bx))$ for all $\bx \in N$ such that $\bv_1, \ldots, \bv_{i-1}, \bx$ are linearly independent.  Consequently,
    $$0 < \trace(Q(\bv_1)) \leq \trace(Q(\bv_2)) \leq \cdots \leq \trace(Q(\bv_n)).$$

\item[(H3)] Let $G = (g_{ij})$ be the Gram matrix of $\{\bv_1, \ldots, \bv_n\}$.  There exist constants $C_1, C_2$, and $C_3$, depending only on $n$ and $F$, such that for any embeddings $\sigma$ and $\tau$ of $F$ into $\mathbb R$,
    \begin{enumerate}
    \item $g_{ii}^\sigma \leq C_1\, g_{jj}^\tau$ for any $1 \leq i \leq j \leq n$ \cite[page 288, (4)]{Humbert},

    \item $\vert g_{ik}^\tau \vert \leq C_2\, g_{ii}^\tau$ for any $1 \leq i \leq k \leq n$, following from \cite[page 279, (1)]{Humbert} and (a),

    \item $\det(G)^\tau \leq g_{11}^\tau \cdots g_{nn}^\tau \leq C_3\, \det(G)^\tau$ \cite[Th\'{e}or\`{e}me 4]{Humbert}.
    \end{enumerate}
\end{enumerate}

For a lack of a more suitable terminology, we simply call $\{\bv_1, \ldots, \bv_n\}$ a Humbert linearly independent set in $N$.  The free $\ring$-lattice $\sum \ring \bv_i$ is called a Humbert sublattice of $N$.  As a consequence of (H1) and (H3), there exists a constant $C_4$, depending only on $n$ and $F$, such that
$$\norm(\vol(N)) \leq \norm(Q(\bv_1))\cdots \norm(Q(\bv_n)) \leq C_4\, \norm(\vol(N)).$$
Therefore, there are only finitely many other $\ring$-lattices on the same space containing the Humbert linearly independent set $\{\bv_1, \ldots, \bv_n\}$.

For each $i$, there are only finitely many choices of $\bv_i$ since $\trace(Q)$ is a positive definite quadratic form on $N$ as a $\z$-lattice.  Thus, there are only finitely many Humbert linearly independent sets in $N$.  From now on, we only work with one choice of Humbert linearly independent set for each positive definite integral  $\ring$-lattice (and hence one choice of Humbert sublattice for each such lattice).  However, we would like to make this choice uniform across the entire isometry class in the following sense.  Fix one positive definite $\ring$-lattice $N$ in its isometry class and select a choice of Humbert linearly independent set $\{\bv_1, \ldots, \bv_n\}$ for $N$.  For each lattice $N'$ in $[N]$, choose an isometry $\varphi: N \longrightarrow N'$ and choose $\{\varphi(\bv_1), \ldots, \varphi(\bv_n)\}$ to be the Humbert linearly independent set for $N'$.  In this way, the Gram matrix with respect to the chosen Humbert linearly independent set is the same for every $\ring$-lattice in $[N]$.  We define $\mu_i(N)$ to be $\trace(Q(\bv_i))$ for $1 \leq i \leq n$.  These are invariants of the isometry class $[N]$.   For the sake of convenience, we also define $\overline{\mu}(N)$ to be $\trace(Q(\bv_n))$.

Let $N'$ be another positive definite $\ring$-lattice of rank $n$ with Humbert linearly independent set $\{\bv_1', \ldots, \bv_n'\}$.  We say that $N$ and $N'$ are {\em Humbert-isomorphic} if the canonical $\ring$-module isomorphism from $\sum \ring \bv_i \longrightarrow \sum \ring \bv_i'$, which sends $\bv_i$ to $\bv_i'$ for $1 \leq i \leq n$, extends to an $\ring$-module isomorphism from $N$ to $N'$.  This defines an equivalence relation on the set of $\ring$-lattices of rank $n$.  Because of our way of fixing the Humbert linearly independent set of an $\ring$-lattice, we see that if $N$ and $N'$ are isometric as $\ring$-lattices, then  $N$ and $N'$ are Humbert-isomorphic.

Let
\begin{enumerate}
\item[(H4)] $t$ be the lcm of all the indices $[N' : \sum \ring\bv_i']$, where $N'$ runs through all positive definite $\ring$-lattices of rank $n$, and $\{\bv_1', \ldots, \bv_n'\}$ runs through all Humbert linear independent sets of $N'$.
\end{enumerate}
This integer $t$ exists by virtue of (H1).  Let $\varphi_{N'} : FN' \longrightarrow F^n$ be the isomorphism such that $\varphi_{N'}(\bv_i') = \be_i$, where $\{\be_1, \ldots, \be_n\}$ is the standard basis of $F^n$.  Then, $\varphi_{N'}(N')$ is an $\ring$-module sitting between $\ring^n$ and $t^{-1}\ring^n$.  Thus, there are only finitely many possibilities of $\varphi_{N'}(N')$.  It is clear that $N$ and $N'$ are Humbert-isomorphic if $\varphi_N(N) = \varphi_{N'}(N')$.  This implies that there are only finitely many Humbert-isomorphism classes of positive definite $\ring$-lattices of a fixed rank.

\subsection{The Escalation Process}

Let $\mathscr S$ be an infinite set of classes of positive definite $\ring$-integral lattices of rank $n$.  As is pointed out in Section 2, $\mathscr S$ is a countable set; so we may fix an enumeration of the classes in $\mathscr S$ as $\{[N_1], [N_2], \ldots \}$.   For any positive definite $\ring$-lattice $L$, define
$$N_L: = \begin{cases}
N_j & \mbox{ if $N_i \longrightarrow L$ $\forall\, i < j$ but $N_j \not \longrightarrow L$};\\
\0 & \mbox{ otherwise}.
\end{cases}$$
For each $i \geq 1$, define a set of classes of $\ring$-lattices $\mathscr T_i$ inductively as follows.  Let $\mathscr T_1$ be the set of classes of positive definite integral lattices of rank $n$ which represent $N_1$.  Clearly, $\mathscr T_1$ is a finite set.  For $i \geq 1$, let
$$\mathscr U_i = \{[L] \in \mathscr T_i : N_L \neq \0\}.$$
If $\mathscr U_i$ is empty (which happens when every class in $\mathscr T_i$ is $\mathscr S$-universal), we set $\mathscr T_{i+1} = \mathscr T_i$.  Otherwise, for each $[L] \in \mathscr U_i$, let $\mathscr U(L)$ be the set of classes of positive definite $\ring$-lattices $M$ such that
\begin{enumerate}
\item[(i)] $M$ represents $L$ and $N_L$, and

\item[(ii)] No sublattice of $M$ of strictly smaller rank can represent both $L$ and $N_L$.
\end{enumerate}
Note that all lattices in any class in $\mathscr U(L)$ have rank $\leq n + \rank(L)$.  Moreover,  the ranks of the classes in $\mathscr U(L)$ is strictly greater than the rank of $L$.  Let
$$\mathscr T_{i+1} = \bigcup_{[L] \in \mathscr U_i} \mathscr U(L).$$

\begin{lem}
For every $i \geq 1$, $\mathscr T_i$ is a finite set and every class of $\ring$-lattices in $\mathscr T_i$ represents $N_j$ for all $j \leq i$.
\end{lem}
\begin{proof}
The second assertion of the lemma follows from the definition of $\mathscr T_i$.  As for the first assertion, we already noted that $\mathscr T_1$ is finite.  Suppose that the assertion holds for some $i \geq 1$.  We may assume that $\mathscr U_i$ is not empty.  Let $[L] \in \mathscr U_i$ and $[M] \in \mathscr U(L)$.  Suppose that $\overline{\mu}(M) > \max\{\overline{\mu}(N_L), \overline{\mu}(L)\}$. Let $K$ be the sublattice of $M$ generated by all $\bx \in M$ such that $\trace(Q(\bx)) < \overline{\mu}(M)$, and let $\tilde{K} = FK\cap M$.  Then, $\tilde{K}$ is a primitive sublattice of $M$ of strictly smaller rank, and $\tilde{K}$ represents the Humbert sublattices of $N_L$ and $L$.  So, $\tilde{K}$ must represent $N_L$ and $L$.  But this is impossible as the rank of $\tilde{K}$ is the same as the rank of $K$ which is strictly smaller than the rank of $M$.  Therefore, $\overline{\mu}(M)$ is bounded above by a constant depending only on $L$.  Since there are only finitely many classes of lattices in $\mathscr U_{i}$ as $\mathscr T_{i}$ is finite, there are only finitely many classes of $M$ in $\mathscr T_{i+1}$.
\end{proof}

\begin{lem} \label{lemmar}
Suppose that there exists $r \geq 1$ such that every class of lattices in $\mathscr T_r$ is $\mathscr S$-universal.  Let
$$\mathscr S_0 = \{[N_1]\} \cup \left\{[N_L] :  [L] \in \bigcup_{i = 1}^{r-1} \mathscr U_i \right \}.$$
Then, every $\mathscr S_0$-universal lattice is $\mathscr S$-universal.
\end{lem}
\begin{proof}
The proof is the same as that of \cite[Proposition 3.1]{kko}.
\end{proof}

All is left to do is to demonstrate the existence of the integer $r$ in this lemma.  In what follows, $\p$ always stands for a finite prime of $F$.

\begin{lem} \label{localfinite}
Let $\mathscr S(\p)$ be an infinite set of $\ring_\p$-lattices of rank $n$.  Then, there exists a finite subset $\mathscr S_0(\p)$ of $\mathscr S(\p)$ such that any $\ring_\p$-lattice that represents every $\ring_\p$-lattice in $\mathscr S_0(\p)$ must represent all $\ring_\p$-lattices in $\mathscr S(\p)$.
\end{lem}
\begin{proof}
The proof is almost identical to \cite[Lemma 1.5]{hkk}.
\end{proof}

The next two lemmas concern representations of $\ring_\p$-lattices.   Their proofs reply on the Jordan decompositions of these lattices which are discussed in detail in \cite[Chapter IX]{om}.  We recall some of the basic properties of these decompositions used in our discussion.

Let $L$ be an $\ring_\p$-lattice.  Then, $L$ has a Jordan decomposition (or Jordan splitting) $L = L_1 \perp \cdots \perp L_t$ in which every orthogonal summand is modular and $\fs(L_1) \supsetneq \cdots \supsetneq \fs(L_t)$.  For $1 \leq i \leq t$, $L_i$ is called the $\fs(L_i)$-component of that decomposition.  The chain of fractional ideals $\fs(L_1) \supsetneq \cdots \supsetneq \fs(L_t)$ is uniquely determined by $L$. For convenience, we denote $\fs(L_t)$, the scale of the last Jordan component, by $\mathfrak l(L)$.  Since $L^\# = L_t^\# \perp \cdots \perp L_1^\#$ and each $L_i^\#$ is $\fs(L_i)^{-1}$-modular, $L_t^\#\perp \cdots \perp L_1^\#$ is a Jordan decomposition of $L^\#$.  In particular, $\mathfrak s(L^\#) = \mathfrak l(L)^{-1}$ and $\mathfrak l(L)^{-1} \supseteq \mathfrak n(L^\#) \supseteq 2\cdot\mathfrak l(L)^{-1}$.

When $\fs(L) = \ring_\p$, the first component in a Jordan decomposition of $L$ must be unimodular.  We call it the unimodular component of that Jordan decomposition.  In general, $L$ has more than one Jordan decompositions.  But when $\p$ is nondyadic, the isometry classes of the Jordan components are uniquely determined.  In particular, when $\p$ is nondyadic and $\mathfrak s(L) = \ring_\p$, the unimodular components in all Jordan decompositions of $L$ are isometric.  In that case, the unimodular component of any Jordan decomposition of $L$ is simply referred as the unimodular component of $L$.

\begin{lem}
Let $K$ be a nondegenerate $\ring_\p$-lattice and $\ideal$ be an ideal contained in $4\cdot \mathfrak l(K)$.  Then, $K$ contains a full-rank $\ideal$-maximal sublattice.  Consequently, $K$ represents all integral $\ring_\p$-lattices of rank $\leq g$ with norm contained in $4\cdot \mathfrak l(K)$, provided $\rank(K) \geq g + 3$.
\end{lem}
\begin{proof}
By \cite[Lemma 5.4]{riehm}, whose proof is for dyadic $\p$ but remains valid for the nondyadic case, $K$ contains an $\ideal$-maximal sublattice on $F_\p K$ as long as $\ideal \subseteq 4\cdot \mathfrak n(K^\#)^{-1}$.  Since $\mathfrak n(K^\#) \subseteq \mathfrak l(K)^{-1}$, therefore
$4\cdot \mathfrak n(K^\#)^{-1} \supseteq 4\cdot\mathfrak l(K)$.   This proves the first assertion.  The rest of the lemma follows from \cite[Lemma 1.1]{hkk} and the remark on \cite[Page 139]{hkk}.
\end{proof}

\begin{lem}
Let $L$ be a nondegenerate integral $\ring_\p$-lattice of rank $\ell \geq 4n + 6$ and $K$ be a primitive sublattice of $L$ of rank $\geq 3n + 6$.  Then, $K$ must represent all integral $\ring_\p$-lattices of rank $n$ with norm contained in $4\cdot \mathfrak l(L)$.
\end{lem}
\begin{proof}
Write $K$ as $K = K_0\perp \cdots \perp K_t$, where $\mathfrak s(K_i) = \p^i$ or $K_i = \0$.   For the sake of convenience, let $m_0 = \ord_\p(\mathfrak l(L))$.  As in \cite[Lemma 2.2]{kko}, we define $K_{+}$ and $K_{-}$ as
$$K_{-} = K_0 \perp \cdots \perp K_{m_0}, \quad K_{+} = K_{m_0 + 1} \perp \cdots \perp K_t.$$
Both $K_{-}$ and $K_{+}$ are primitive sublattices of $L$.  Since $\mathfrak l(K_{-}) \supseteq \mathfrak l(L)$, it suffices show that $\rank(K_{-}) \geq n + 3$ by virtue of the last lemma.

Let $r$ be the rank of $K_{+}$, and $\{\bx_1, \ldots, \bx_r\}$ be a basis of $K_{+}$.  Extend $\{\bx_1, \ldots, \bx_r\}$ to a basis $\{\bx_1, \ldots, \bx_r, \by_1, \ldots, \by_{\ell - r}\}$ of $L$.

Assume for the contradiction that $2r > \ell$.  Then, $\ell - r < r$ and the rows of the $r\times (\ell - r)$ matrix $(B(\bx_i, \by_j))$ must be linearly dependent over $F_\p$.  Thus, there must be a primitive row vector $(a_1, \ldots, a_r)$ of $\ring_\p^r$ such that
$$(a_1, \ldots, a_r)\cdot (B(\bx_i, \by_j)) = \0.$$
Let $\bv = a_1\bx_1 + \cdots + a_r\bx_r$.  Then, $B(\bv, L) = \{0\}$ mod $\p^{m_0 + 1}$.  However, since $\bv$ is a primitive vector of $K_{+}$ which is a primitive sublattice of $L$, $\bv$ is also a primitive vector of $L$.  Since $\mathfrak l(L) = \p^{m_0}$,  $L$ has a Jordan decomposition $L = L_1 \perp \cdots \perp L_t$ with $\mathfrak s(L_i) \supseteq \p^{m_0}$ for all $i$.  Then, $\bv = \bv_1 + \cdots + \bv_t$ with $\bv_i \in L_i$ for all $i$ and $\bv_{j}$ primitive in $L_j$ for at least one $j$.  As a result, $B(\bv, L) \supseteq B(\bv, L_j) = B(\bv_j, L_j) = \mathfrak s(L_j) \supseteq \p^{m_0}$ by virtue of \cite[82:17]{om}. This is a contradiction.

So, we must have $2r \leq \ell$ and
$$\rank(K_{-}) = \rank(K) - r \geq \ell - n - r \geq \frac{\ell}{2} - n \geq n + 3$$
which is what we need.
\end{proof}

\begin{lem} \label{q}
Let $L$ be a positive definite $\ring$-lattice of rank $\geq 4n + 6$, and $t$ be a positive integer.  There exists an ideal $\qq$ of $\ring$, depending only on $tL$ with the following properties: let $N$ be a lattice of rank $n$, $M$ be a lattice which represents both $N$ and $L$, and $\tau: L \longrightarrow M$ and $\sigma: N \longrightarrow M$ be representations.  Then the genus of $\sigma(N)^\perp\cap \tau(tL)$, where $\sigma(N)^\perp$ is the orthogonal complement of $\sigma(N)$ in $M$, represents all positive definite $\ring$-lattices of rank $n$ with norm contained in $\qq$.
\end{lem}
\begin{proof}
Without loss of generality, we may assume that both $N$ and $L$ are sublattices of $M$, and that $\sigma$ and $\tau$ are the natural inclusions.  Let $\pr : FM \longrightarrow FL$ be the orthogonal projection from $FM$ onto $FL$.  Take an nonzero element $\alpha$ of $\ring$ such that $\tilde{N} : = \pr(\alpha N)$ is a sublattice of $tL$.  Let $K$ be the orthogonal complement of $\tilde{N}$ in $tL$, which is a primitive sublattice of $tL$ of rank $\geq 3n + 6$.

Let $\bv$ be a vector in $K$, and $\bx$ a vector in $N$.  Then, $\by: = \pr(\alpha \bx)$ is in $\tilde{N}$, and $\alpha \bx = \by + \bz$ for some $\bz$ orthogonal to $FL$.  Clearly, $B(\bv, \by) = B(\bv, \bz) = 0$.  Therefore, $B(\bv, \alpha \bx) = 0$; hence $B(\bv, \bx) = 0$ as well.  This shows that $K \subseteq N^\perp$.

Let
$$\qq = \prod_{\p} \p^{\ord_\p(4\cdot\mathfrak l(tL_\p))}.$$
This ideal $\qq$ is well-defined because $\ord_\p(4\cdot\mathfrak l(tL_\p))$ is 0 for almost all $\p$.  The conclusion of the lemma now follows from the last lemma.
\end{proof}

\subsection{The Finiteness Theorem}

We are now ready to complete the proof of number-field analog of the main theorem in \cite{kko}.

\begin{thm} \label{finite}
Let $\mathscr S$ be an infinite set of isometry classes of positive definite $\ring$-lattices of a fixed rank.  Then, there exists a finite subset $\mathscr S_0$ of $\mathscr S$ such that if a positive definite $\ring$-lattice represents all classes of lattices in $\mathscr S_0$, then it represents all classes of lattices in $\mathscr S$.
\end{thm}
\begin{proof}
It suffices to show that an integer $r$ satisfying the hypothesis of Lemma \ref{lemmar} does exist.  Assume to the contrary that no such integer exists.  Then, for every positive integer $i$, the set $\mathscr U_i$ is nonempty, and the ranks of all the classes in $\mathscr T_{i+1}$ is strictly greater than the ranks of all the classes in $\mathscr T_{i}$.  Therefore, there must be a positive integer $s$ such that the ranks of all the classes in $\mathscr T_s$ are $\geq n + 3$.

Let $P$ be the set of all finite primes $\p$ of $F$ such that
\begin{enumerate}
\item[(i)] $P$ contains all the dyadic primes and all the prime divisors of $t$, where $t$ is the integer defined in (H4) for rank $n$.

\item[(ii)] $L_\p$ is unimodular for all $\p \not \in P$ and all $L \in \mathscr T_s$.
\end{enumerate}
Clearly, this $P$ is a finite set.

By Lemma \ref{localfinite}, there are finitely many classes $[N^{(j)}]$ in $\mathscr S$ such that for each $\p \in P$, if an $\ring_\p$-lattice represents $N_\p^{(j)}$ for all $j$, then it represents $N_\p$ for all $[N] \in \mathscr S$.  We take an integer $s_1 \geq s$ such that
\begin{enumerate}
\item[(iii)] the ranks of all the classes of lattices in $\mathscr T_{s_1}$ are $\geq 4n + 6$,

\item[(iv)] $N_\p \longrightarrow L_\p$ for all $\p \in P$, all $[N] \in \mathscr S$, and all $[L] \in \mathscr T_{s_1}$,

\item[(v)] $L_\p$ has a unimodular component of rank $\geq n + 3$ for all $\p \not \in P$ and all $[L] \in \mathscr T_{s_1}$.
\end{enumerate}
Condition (v) is a consequence of condition (ii) and the fact that every class of lattices in $\mathscr T_{s+1}$ represents some classes of lattices in $\mathscr T_{s}$.

For every finite $\p \not \in P$, $\p$ is nondyadic and $L_\p$ has a unimodular component of rank $\geq n+3$ for every $[L] \in \mathscr T_{s_1}$.  By \cite[Theorem 1]{omr}, that unimodular component alone represents all $\ring_\p$-lattices of rank $n$.  Therefore, conditions (iv) and (v) together ensure that for each finite prime $\p$ of $F$,  $N_\p$ is represented by $L_\p$ for all $[N] \in \mathscr S$ and all $[L] \in \mathscr T_{s_1}$.  By \cite[Theorem 3]{hkk}, there exists $c_1 > 0$ such that every class of lattices in $\mathscr T_{s_1}$ represents all classes $[N]$ in $\mathscr S$ provided $\mu_1(N) > c_1$.  For each $[L] \in \mathscr T_{s_1}$, let $\qq(tL)$ be the ideal of $\ring$ satisfying Lemma \ref{q} for the lattice $L$ and the integer $t$, and let
$$\qq: = \prod_{[L] \in \mathscr T_{s_1}} \qq(tL).$$

Let $1 \leq i \leq n-1$.  Suppose that there exists $c_i \geq c_1$ and $s_i \geq s_1$ such that every class of lattices $[M]$ in $\mathscr T_{s_i}$ represents all $[N] \in \mathscr S$ provided $\mu_i(N) > c_i$.  Let
$$\mathscr S_i: = \{[N] \in \mathscr S : N \not \longrightarrow M \mbox{ for some } [M] \in \mathscr T_{s_i}\}.$$
Then, for every  $[N] \in \mathscr S_i$, $\mu_i(N) \leq c_i$.  So, the Gram matrix $(g_{ij})$ of the Humbert linearly independent set of $N$ satisfies
$$\trace(g_{11}) \leq \cdots \leq \trace(g_{ii}) \leq c_i.$$
Since every $g_{jj}$ is totally positive, we must have only finitely many possibilities of $g_{11}, \ldots, g_{ii}$.  By (H3), we see that there are also finitely many possibilities of $g_{jk}$ for any $1 \leq j \leq i$ and $1 \leq k \leq n$.  So, $\mathscr S_i$ is partitioned into finitely many disjoint parts in such a way that if $\{\bv_1, \ldots, \bv_n\}$ and $\{\bv_1', \ldots, \bv_n'\}$ are the Humbert linearly independent sets of two lattices from different classes in the same part, then their Gram matrices will satisfy
$$B(\bv_j , \bv_k) = B(\bv_j', \bv_k') \quad \mbox{ for } 1\leq j \leq i \mbox{ and } 1\leq k \leq n.$$
Each of these parts is the disjoint union of finitely many subsets such that all the lattices in the classes within any of these subsets are Humbert isomorphic.  Let $\mathscr X$ be one of these subsets.  Let $[N]$ and $[N']$ be two classes in $\mathscr X$, and suppose $\{\bv_1, \ldots, \bv_n\}$ and $\{\bv_1', \ldots, \bv_n'\}$ are the Humbert linearly independent sets of $N$ and $N'$, respectively.  We write $[N] \sim [N']$ when
$$B(\bv_j, \bv_k) \equiv B(\bv_j', \bv_k') \mod \qq \quad \mbox{ for } i+1 \leq j \leq k \leq n.$$
Then, $\sim$ is an equivalence relation on $\mathscr X$, and the resulting quotient set has only finitely many equivalence classes.  Let $\mathscr W$ be a complete set of representatives of these equivalence classes, which is a finite set, and let $a = a(\mathscr X)$ be an integer $\geq s_i$ such that every class of lattices in $\mathscr T_a$ represents all classes of lattices in $\mathscr W$.

Let $[N] \in \mathscr X$ and $[\overline{N}]$ be the class in $\mathscr W$ such that $[N] \sim [\overline{N}]$, and let $\{\bv_1, \ldots, \bv_n\}$ and $\{\overline{\bv}_1, \ldots, \overline{\bv}_n\}$ be the Humbert linearly independent sets of $N$ and $\overline{N}$, respectively.  Let $A$ be the lattice corresponding to the $(n-i)\times (n-i)$ symmetric matrix $(B(\bv_j, \bv_k)) - (B(\overline{\bv}_j, \overline{\bv}_k))$, where $j, k \in \{i+1,  \ldots, n\}$.  When $\mu_{i+1}(N)$ is sufficiently large,  $A$ would be positive definite of rank $n - i$.

Let $[M] \in \mathscr T_a$.  Then, $M$ represents $\overline{N}$ and $L$ for all $[L] \in \mathscr T_{s_1}$.  We may assume that both $\overline{N}$ and $L$ are sublattices of $M$ and that the representations of $\overline{N}$ and $L$ into $M$ are the natural inclusions.   When $\mu_{i+1}(N)$ is sufficiently large, $A$ is positive definite and is represented by $\overline{N}^\perp \cap tL \subseteq \overline{N}^\perp \cap tM$ by Lemma \ref{q} and \cite[Theorem 3]{hkk}.  Here, $\overline{N}^\perp$ is the orthogonal complement of $\overline{N}$ in $M$.   In that case, there exist $\bw_{i+1}, \ldots, \bw_n$ in $\overline{N}^\perp \cap tM$ such that
$$(B(\bw_j, \bw_k)) = B(\bv_j, \bv_k)) - (B(\overline{\bv}_j, \overline{\bv}_k)).$$
We set $\bw_j = \0$ for $j = 1, \ldots, i$.   Then,  the map $\varphi: \sum \ring \bv_j \longrightarrow M$ defined by
$$\bv_j \longmapsto \overline{\bv}_j + \bw_j, \quad j = 1, \ldots, n$$
is a representation from $\sum \ring \bv_j$ into $M$.  Take $\bv \in N$.  Then, $t\bv$ is in $\sum \ring\bv_j$, hence $\varphi(t\bv) = \overline{\bv} + \bw$ where $\overline{\bv} \in \sum \ring\overline{\bv}_j$ and $\bw \in \overline{N}^\perp \cap tM$.  Since $[N]$ and $[\overline{N}]$ are Humbert-isomorphic, $t^{-1}\overline{\bv}$ must be in $\overline{N}$.  As for $\bw$, it is in $tM$ and hence $t^{-1}\bw$ is in $M$.  Therefore, $\varphi(t\bv)$ is in $tM$, rendering $\varphi(\bv)$ in $M$.  Consequently, $\varphi$ is a representation from $N$ into $M$.

We have demonstrated that there exists a positive integer $c(\mathscr X)$ such that if $[N] \in \mathscr X$ and $\mu_{i+1}(N) > c(\mathscr X)$, then $N$ is represented by all classes of lattices in $\mathscr T_a$.  Since $\mathscr S_i$ is the disjoint union of finitely many these $\mathscr X$, we may conclude that there exist positive integers $s_{i+1} \geq s_i$ and $c_{i+1}$ such that every class of lattices in $\mathscr T_{s_{i+1}}$ represents all $[N] \in \mathscr S$ provided $\mu_{i+1}(N) > c_{i+1}$.   We may continue this argument until we obtain integers $s_n$ and $c_n$ such that every class of lattices in $\mathscr T_{s_n}$ represents all classes $[N] \in \mathscr S$ provided $\overline{\mu}(N) > c_n$.  This implies that every class of lattices in $\mathscr T_{s_n}$ represents all but finitely many classes in $\mathscr S$.  So, we could find a large enough integer $r$ such that every class of lattices in $\mathscr T_r$ is $\mathscr S$-universal.  This is a contradiction.
\end{proof}

A quadratic form $f$ in $n$ variables over $F$ is {\em integer-valued} if $f(\ring^n) \subseteq \ring$.  If $f$ is positive definite and $f(\ring^n)$ is the set of totally positive integers in $F$, then $f$ is called universal.  It was proved in \cite{bh} that a positive definite integer-valued quadratic form over $\mathbb Q$ is universal if it represents every positive integers up to 290.  The number-field extension of this result has been sought after in some recent work on universal quadratic forms (see, for example, \cite{k} and \cite{ky}).  More precisely, they seek a finite set $\mathscr K_0$ of totally positive integers of $F$ such that a positive definite integer-valued quadratic form over $F$ must be universal if it represents all elements of $\mathscr K_0$.  We call such a set $\mathscr K_0$ a finite universality  criterion set for $F$.
The following is an immediate consequence of Theorem \ref{finite}.

\begin{cor}
Let $F$ be a totally real number field.  There is a finite universality  criterion set for $F$.
\end{cor}

\end{document}